\documentclass{amsart}
\usepackage[T1]{fontenc}   
\usepackage{amssymb}
\usepackage{amsmath}
\usepackage{amsthm}
\usepackage{enumerate}
\usepackage[urlcolor=black, colorlinks=true, citecolor=black, linkcolor=black]{hyperref}
\usepackage[english]{babel}
\usepackage{dsfont} 

\newtheorem{myclm}{Claim}
\newtheorem{mythm}{Theorem} \numberwithin{mythm}{section} 

\newtheorem{mycor}[mythm]{Corollary}
\newtheorem{mylem}[mythm]{Lemma} 
\newtheorem{myquest}[mythm]{Question}

\newcommand{ \N } { \mathbb{N} }

\newcommand{\w}{\omega}
\newcommand{\wstar}{\omega^*}
\newcommand{\Cstar}{C^*}
\newcommand{\cont}{\mathfrak{c}}

\newcommand{\script}{\mathcal}
\newcommand{\parentheses}[1]{{\left( {#1} \right)}}
\newcommand{\p}{\parentheses}

\newcommand{\closure}[1]{\overline{#1}}
\newcommand{\closureIn}[2]{\closure{#1}^{#2}}

\newcommand{\Set}[1]{{\left\lbrace {#1} \right\rbrace}}
\newcommand{\singleton}{\Set}

\newcommand{\Union}{\bigcup}

\newcommand{\cardinality}[1]{{\left\lvert {#1} \right\rvert}}

\newcommand{\clopens}[1]{\script{CO}\p{#1}}
\newcommand{\cut}[1]{\script{C}\p{#1}}
\newcommand{\pair}[1]{\langle {#1} \rangle}
\def\set#1:#2{\Set{{#1} \colon {#2}}}

\newcommand{\continuum}{\mathfrak{c}}
\renewcommand{\subset}{\subseteq}
\renewcommand{\supset}{\supseteq}

\begin{document}


\renewcommand{\bf}{\bfseries}
\renewcommand{\sc}{\scshape}

\title{Compactifications of $\wstar \setminus \singleton{x}$ and $S_\kappa \setminus \singleton{x}$}

\author[M.\,Pitz]{Max F.\ Pitz}
\address[Pitz]{Mathematical Institute\\University of Oxford\\Oxford OX2 6GG\\United Kingdom}
\email[Corresponding author]{pitz@maths.ox.ac.uk}
\thanks{The first author acklowledges support of the German Academic Exchange Service (DAAD)}

\author[R.\,Suabedissen]{Rolf Suabedissen}
\address[Suabedissen]{Mathematical Institute\\University of Oxford\\Oxford OX2 6GG\\United Kingdom}
\email{suabedis@maths.ox.ac.uk}



\begin{abstract}
The title refers to the Stone-\v{C}ech remainder of the integers $\wstar$ and the Stone space $S_\kappa$ of the $\kappa$-saturated Boolean algebra of cardinality $\kappa$. The latter space is characterised topologically as the unique $\kappa$-Parovi\v{c}enko space of weight $\kappa$. It exists if and only if the consistent and independent equality $\kappa = \kappa^{<\kappa}$ holds. The spaces $S_\kappa$ are generalisations of the space $\wstar$: under the Continuum Hypothesis, $S_{\w_1}$ is homeomorphic to $\wstar$.

In this paper we investigate compactifications of spaces $S_\kappa \setminus \singleton{x}$, building on and extending corresponding results obtained by Fine \& Gillman and Comfort \& Negrepontis for the space $\wstar$. 

We show that for every point $x$ of $S_\kappa$, the Stone-\v{C}ech remainder of $S_\kappa \setminus \singleton{x}$ is a $\kappa^+$-Parovi\v{c}enko space of cardinality $2^{2^\kappa}$ which admits a family of $2^\kappa$ disjoint clopen sets. As a corollary we get that it is consistent with CH that the Stone-\v{C}ech remainders of $\omega^* \setminus \{x\}$ are all homeomorphic.
\end{abstract}

\maketitle

\section{\bf Introduction}
The purpose of this paper is to study compactifications of subspaces of the Stone-\v{C}ech remainder of the integers $\wstar = \beta \w \setminus \w$, namely subspaces of the form $\wstar \setminus \singleton{x}$. The questions discussed in this paper are motivated by two classical theorems about extensions of real-valued continuous functions on $\wstar \setminus \singleton{x}$.

\begin{mythm}[Fine and Gillman, 1960 \cite{fine,Gillmann}]
\label{thm1.1}
Assuming the Continuum Hypothesis, for every point $x \in \wstar$ there are continuous bounded real-valued functions on $\wstar \setminus \singleton{x}$ that cannot be continuously extended to $\wstar$.
\end{mythm}

\begin{mythm}[van Douwen, Kunen and van Mill, 1989 \cite{douwenkunenmill}]
\label{thm1.2}
It is consistent with the continuum being $\aleph_2$ that for every point $x \in \wstar$ all continuous real-valued functions on $\wstar \setminus \singleton{x}$ can be continuously extended to $\wstar$.
\end{mythm}

Phrased in the language of compactifications, Theorem \ref{thm1.2} says that the Stone-\v{C}ech compactification of $\wstar \setminus \singleton{x}$ consistently coincides with $\wstar$. Theorem \ref{thm1.1}, however, guarantees that under the Continuum Hypothesis (CH), the Stone-\v{C}ech compactification certainly adds more than a single point to $\wstar \setminus \singleton{x}$. 

This last observation gives rise to a variety of new, interesting questions. What are the precise mechanisms that increase the size of the Stone-\v{C}ech compactification of $\wstar \setminus \singleton{x}$ under CH? Can one determine its size or is this independent of ZFC+CH? Does the Stone-\v{C}ech remainder of $\wstar \setminus \singleton{x}$ reflect structural properties of $\wstar$ or $\wstar \setminus \singleton{x}$ and in particular, does this depend on which point $x$ we remove?

What makes these problems particularly interesting is a surprising link between compactifications of $\wstar \setminus \singleton{x}$ and another research area in the vicinity of $\wstar$: the theory of $\kappa$-Parovi\v{c}enko spaces, developed by Negrepontis  \cite{negrepontis}, Comfort and Negrepontis \cite{Ultrafilters} and Dow \cite{Dow}.

A rigorous introduction to $\kappa$-Parovi\v{c}enko spaces will follow in the next section. Intuitively, however, $\kappa$-Parovi\v{c}enko spaces generalise crucial characteristics of $\wstar$ to spaces of (potentially larger) weight $\kappa$. Recall that under CH, or equivalently under $\omega_1 = \w_1^{<\w_1}$, the space $\wstar$ is topologically characterised as the unique Parovi\v{c}enko space of weight $\continuum$.  Here, the Parovi\v{c}enko properties entail a precise description of the behaviour of countable unions and intersections of clopen sets in $\wstar$: disjoint pairs of such countable unions have disjoint closures, and every such non-empty  intersection has non-empty interior. The $\kappa$-Parovi\v{c}enko properties essentially consist of corresponding requirements for all $\lambda$-unions and $\lambda$-intersections of clopen sets for all $\lambda < \kappa$. And as for $\wstar$ under CH, for every cardinal $\kappa$ with the property $\kappa = \kappa^{<\kappa}$ there is a unique $\kappa$-Parovi\v{c}enko space of weight $\kappa$, denoted by $S_\kappa$. In particular, $S_{\w_1}$ coincides with $\wstar$ under CH.

The main result of this paper is that under CH, the Stone-\v{C}ech remainder of $\wstar \setminus \singleton{x}$ is an $\w_2$-Parovi\v{c}enko space of cardinality $2^{2^{\w_1}}$ and contains a family of $2^{\w_1}$ disjoint open sets, regardless of the choice of $x$. In general, under $\kappa=\kappa^{<\kappa}$, the Stone-\v{C}ech remainder of $S_\kappa \setminus \singleton{x}$ is a $\kappa^+$-Parovi\v{c}enko space of cardinality $2^{2^\kappa}$ and contains a family of $2^\kappa$ disjoint clopen sets. Again, this holds for every point $x$ in $S_\kappa$. 					

As a consequence, assuming $2^\cont=\w_2$, the Stone-\v{C}ech remainder of any $\wstar \setminus \singleton{x}$ is homeomorphic to the unique $\w_2$-Parovi\v{c}enko space of weight $\w_2$, independently of which point $x$ gets removed. This is surprising, considering that subspaces $\wstar \setminus \singleton{x}$ and $\wstar \setminus \singleton{y}$ are typically very different.

Our main result also improves a theorem by A.\,Dow, who showed in 1985 that the remainder of $S_\kappa \setminus \singleton{p}$ is a $\kappa^+$-Parovi\v{c}enko space provided that $p$ has a nested neighbourhood base in $S_\kappa$, \cite{Dow}. Consequently, we now have a more comprehensive answer to a question by S.\,Negrepontis \cite[p.\,522]{negrepontis}: under $\kappa^+=2^\kappa$, the space $S_{\kappa^+}$ can be represented as the Stone-\v{C}ech remainder of any $S_\kappa \setminus \singleton{x}$, regardless of whether $x$ has a nested neighbourhood base or not. 

This paper is organised as follows. In Section \ref{section2} we recall the relevant background for the investigation of $\wstar$ and $S_\kappa$. Section \ref{section3} provides a short account of Fine and Gillman's result that under CH, no $\wstar \setminus \singleton{x}$ is $\Cstar$-embedded in $\wstar$, and concludes with a characterisation of the clopen subsets of $\wstar \setminus \singleton{x}$. Section \ref{section4} generalises these results to spaces $S_\kappa$ for any $\kappa$ with $\kappa=\kappa^{<\kappa}$. In Section \ref{section6} we show that the Stone-\v{C}ech compactification of $S_\kappa \setminus \singleton{x}$ is a $\kappa$-Parovi\v{c}enko space of weight $2^\kappa$. 

The next two sections are concerned with the remainder of $S_\kappa \setminus \singleton{x}$. First we show in Section \ref{section7} that  all remainders are $\kappa^+$-Parovi\v{c}enko spaces and conclude that under additional set theoretic assumptions they are all homeomorphic. In Section \ref{section8} we answer some questions about cardinal invariants of these spaces and prove that the Stone-\v{C}ech remainder of $S_\kappa \setminus \singleton{x}$ has cardinality $2^{2^\kappa}$. Our proof builds on the observation that under CH, the space $\wstar$ is a \emph{monotone F-space}, i.e.\ that it is possible to assign a separating clopen partition to pairs of disjoint open $F_\sigma$-sets, in the spirit of a monotone normality operator. We also give an explicit topological construction of a family of $2^\kappa$ disjoint open sets in the Stone-\v{C}ech remainder of $S_\kappa \setminus \singleton{x}$.

Each of Section \ref{section6} to \ref{section8} builds only on Section \ref{section4} and can be read independently. Section \ref{section9} concludes the paper with some open questions. 

The authors would like to thank Alan Dow for his help with the proof of Theorem \ref{hui} at the Spring Topology and Dynamics Conference 2014 in Richmond, Virginia.

\section{\bf Background}
\label{section2}

We recall properties of $\wstar$ and $S_\kappa$ which we use freely in later sections. All this and more can be found in \cite{Ultrafilters,Rings, Intro}. The reader should note that our notion of a ($\kappa$-)Parovi\v{c}enko space differs from the commonly adopted definition in the sense that it does not a priori include any assumptions regarding weight.

A subset $Y$ of a space $X$ is $\Cstar$-\emph{embedded} if every bounded real-valued continuous function on $Y$ can be continuously extended to all of $X$. For a Tychonoff space $X$, we write $\beta X$ for its \emph{Stone-\v{C}ech compactification}, the unique compact Hausdorff space in which $X$ is densely $\Cstar$-embedded, and we write $X^*=\beta X \setminus X$ for the \emph{remainder} of $X$. A space is \emph{zero-dimensional} if it has a base of clopen (closed-and-open) sets. A space $X$ is \emph{strongly zero-dimensional} if its Stone-\v{C}ech compactification $\beta X$ is zero-dimensional. A Tychonoff space $X$ is called $F$-\emph{space} if each cozero-set is $\Cstar$-embedded in $X$. We list some facts about $F$-spaces \cite[Ch.\,14]{Rings}.

\begin{enumerate}
\item $X$ is an $F$-space if and only if $\beta X$ is an $F$-space.
\item A normal space is an $F$-space if and only if disjoint open $F_\sigma$-sets have disjoint closures.
\item Closed subspaces of normal $F$-spaces are again $F$-spaces.
\item Infinite closed subspaces of compact $F$-spaces contain a copy of $\beta \w$. Therefore, compact $F$-spaces do not contain convergent sequences.
\end{enumerate}

In compact zero-dimensional spaces, cozero-sets are countable unions of clopen sets. This motivates the following definition from \cite[Ch.\,14]{Ultrafilters}. In a zero-dimensional space $X$, the ($X$-)\emph{type} of an open subset $U$ of $X$ is the least cardinal number $\tau=\tau(U)$ such that $U$ can be written as a union of $\tau$-many clopen subsets of $X$. A zero-dimensional space where open subsets of type less than $\kappa$ are $\Cstar$-embedded is called $F_\kappa$-\emph{space}. In zero-dimensional compact spaces the notions of $F$- and $F_{\w_1}$-space coincide.

\begin{enumerate}
\item[($1'$)] $X$ is a strongly zero-dimensional $F_\kappa$-space if and only if $\beta X$ is a strongly zero-dimensional $F_\kappa$-space. See Theorem \ref{FKappaSpaces}.
\item[($2'$)] In $F_\kappa$-spaces, disjoint open sets of types less than $\kappa$ have disjoint closures, and in normal spaces both conditions are equivalent, \cite[6.5]{Ultrafilters}.
\end{enumerate}

The space $\wstar$ is a compact zero-dimensional Hausdorff space without isolated points with the following two extra properties: it is an $F$-space in which each non-empty $G_\delta$-set has non-empty interior. Such a space is also called \emph{Parovi\v{c}enko space}. 

We call a space a $G_\kappa$\emph{-space} if every non-empty intersection of less than $\kappa$-many open sets has non-empty interior. Then $\wstar$ is a $G_{\w_1}$-space.

\begin{enumerate}
\setcounter{enumi}{4}
\item The following are equivalent for a zero-dimensional space $X$:
\begin{enumerate}
\renewcommand{\labelenumi}{(\alph{enumi})}
\item $X$ is a $G_\kappa$-space,
\item For every open subset $U$ of $X$ of type less than $\kappa$, no set $H$ with $U \subsetneq H \subset \closure{U}$ is open,
\item For every open subset $U$ of $X$ such that $1 < \tau(U) < \kappa$, its closure $\closure{U}$ is not open,
\item No open subset $U$ of $X$ with $1 < \tau(U) < \kappa$ is dense in $X$.
\end{enumerate}
\end{enumerate}

For proofs of the non-trivial implications $(a) \Rightarrow (b)$ and $(d) \Rightarrow (a)$ see \cite[14.5]{Ultrafilters} and \cite[6.6]{Ultrafilters} respectively.

We call a space a $\kappa$-\emph{Parovi\v{c}enko space} if it is a compact zero-dimensional $F_\kappa$- and $G_\kappa$-space without isolated points. This modifies the corresponding definitions of \cite{douwenmill} and \cite[1.2]{Dow}, freeing $\kappa$-Parovi\v{c}enko spaces from additional weight restrictions. Under our definition, the space $\wstar$ is an ($\w_1$-)Parovi\v{c}enko space of weight $\cont$.

It is not hard to prove that any $\kappa$-Parovi\v{c}enko space has weight at least $\kappa^{<\kappa}$. Thus, the following two  characterisation theorems say, loosely speaking, that the smallest possible ($\kappa$-)Parovi\v{c}enko spaces, namely the ones of weight $\kappa$, are topologically unique, and larger ones are not. 

\begin{mythm}[Parovi\v{c}enko {\cite{parov}}, van Douwen \& van Mill {\cite{douwenmill}}]
CH is equivalent to the assertion that every ($\w_1$-)Parovi\v{c}enko space of weight $\cont$ is homeomorphic to $\wstar$.
\end{mythm}
\begin{mythm}[Negrepontis {\cite{negrepontis}}, Dow {\cite{Dow}}]
The assumption $\kappa = \kappa^{<\kappa}$ is equivalent to the assertion that all $\kappa$-Parovi\v{c}enko spaces of weight $\kappa^{<\kappa}$ are homeomorphic.
\end{mythm}

If the condition $\kappa=\kappa^{<\kappa}$ is satisfied then the unique $\kappa$-Parovi\v{c}enko space of weight $\kappa$ exists and is denoted by $S_\kappa$ \cite[6.12]{Ultrafilters}. Whenever the space $S_{\w_1}$ exists, it is homeomorphic to $\wstar$. The existence of uncountable cardinals satisfying the equality $\kappa=\kappa^{<\kappa}$ is independent of ZFC but an assumption like $\kappa^+=2^\kappa$ implies the equality for $\kappa^+$. Also note that $\kappa=\kappa^{<\kappa}$ implies that $\kappa$ is regular.

A $P_\kappa$\emph{-point} $p$ is a point such that the intersection of less than $\kappa$-many neighbourhoods of $p$ contains again an open neighbourhood of $p$. A $P_{\w_1}$-point is simply called $P$-point. That $P_\kappa$-points exist in $S_\kappa$ is well-known \cite[6.17]{Ultrafilters}. A new proof of this fact is contained in Corollary \ref{mycor12}. We list some facts about $P_\kappa$-points.
\begin{enumerate}
\setcounter{enumi}{5}
\item In $S_\kappa$, $p$ is a $P_\kappa$-point if and only if $p$ has a nested neighbourhood base if and only if $p$ is not contained in the boundary of any open set of type less than $\kappa$.  
\item For every pair of $P_\kappa$-points in $S_\kappa$ there exists an autohomeomorphism of $S_\kappa$ mapping one $P_\kappa$-point to the other \cite[6.21]{Ultrafilters}.
\end{enumerate}
In particular, the subspaces $S_\kappa \setminus \singleton{p}$ are homeomorphic for all $P_\kappa$-points $p$. Corresponding results hold for $P$-points in $\wstar$ under CH.

\section{\bf The Butterfly Lemma}
\label{section3}

We begin with the classic result that under CH, $\wstar$ does not occur as the Stone-\v{C}ech compactification of any of its dense subspaces.

A point $x$ of a Hausdorff space $X$ is called \emph{strong butterfly point} if its complement $X \setminus \singleton{x}$ can be partitioned into open sets $A$ and $B$ such that $\closure{A} \cap \closure{B} = \singleton{x}$. The sets $A$ and $B$ are called \emph{wings} of the butterfly point $x$. Note that in $X\setminus \singleton{x}$, the wings $A$ and $B$ are clopen and non-compact.

The following lemma by Fine and Gillman \cite{fine,Gillmann} states that under CH, every point in $\wstar$ is a strong butterfly point. We outline the proof, as later on we will use variations of this approach in Theorem \ref{hui} and Lemmas \ref{moretechnicalbits} and \ref{cell}.

\begin{mylem}[Butterfly Lemma, Fine and Gillman]
\label{butterfly}
\textnormal{[CH].} Every point in $\wstar$ is a strong butterfly point.
\end{mylem}

\begin{proof}
Let $x \in \wstar$ and fix a neighbourhood base $\set{U_\alpha}:{\alpha < \omega_1}$ of $x$ consisting of clopen sets. By transfinite recursion we define families of clopen sets $\set{A_\alpha}:{\alpha < \omega_1}$ and $\set{B_\alpha}:{\alpha < \omega_1}$ not containing $x$ such that all ($A_\alpha$,$B_\beta$)-pairs are disjoint and 
$$ X \setminus U_\alpha \subset A_\alpha \cup B_\alpha \; \text{ and } \; A_\alpha \cap U_\alpha \neq \emptyset \neq B_\alpha \cap U_\alpha \;  \text{ for all} \; \alpha < \w_1.$$
Once the construction is completed, we define disjoint open sets
$$A = \Union_{\alpha < \omega_1} A_\alpha \quad \text{and} \quad B = \Union_{\alpha < \omega_1} B_\alpha.$$
Their union covers of all of $\wstar \setminus \singleton{x}$ and both $A$ and $B$ limit onto $x$. 

It remains to describe the recursive construction. Let $\alpha < \w_1$ and assume that $A_\beta$ and $B_\beta$ have been defined for all  ordinals $\beta < \alpha$. Since countable unions of clopen sets are open $F_\sigma$-sets, by the $F$-space property there exist  clopen sets $C$ and $D$ partitioning $\wstar$ and containing the disjoint sets $\Union_{\beta < \alpha} A_\beta$ and $\Union_{\beta < \alpha} B_\beta$ respectively. 

The set $U_\alpha \setminus \Union_{\beta < \alpha} (A_\beta \cup B_\beta)$ is a non-empty $G_\delta$-set of the Parovi\v{c}enko space $\wstar$ as it contains $x$ and thus, it has non-empty interior. Hence, inside this set we may find disjoint non-empty clopen sets $C'$ and $D'$ not containing $x$. By defining $A_\alpha = (C \setminus U_\alpha) \cup C'$ and $B_\alpha = (D \setminus U_\alpha) \cup D'$ we see that $A_\alpha$ and $B_\alpha$ are as required. 
\end{proof}

We list some consequences of the Butterfly Lemma regarding $\wstar$ under CH. First must come the result for which the Butterfly Lemma was originally invented. 

\begin{mycor}[Fine and Gillman]
\label{nonStone}
\textnormal{[CH].} For every point $x$ of $\wstar$ the subspace $\wstar \setminus \singleton{x}$ is not $\Cstar$-embedded in $\wstar$. \qed
\end{mycor}

The construction of the Butterfly Lemma can be used to build $2^{\w_1}$ many distinct clopen subsets of $\wstar \setminus \singleton{x}$. Indeed, in the proof of Lemma \ref{butterfly}, for every $\alpha < \w_1$ we have the choice of adding $C'$ or $D'$ to $A_\alpha$. The collection of all clopen subsets of a space $X$ is denoted by $\clopens{X}$. 

\begin{mycor}
\label{carclopens}
\textnormal{[CH].} For all $x$ in $\wstar$ we have $\cardinality{\clopens{\wstar \setminus \singleton{x}}} = 2^{\w_1}$. \qed
\end{mycor} 

Compact clopen sets of $\wstar \setminus \singleton{x}$ are of course homeomorphic to $\wstar$. The next lemma describes how the non-compact clopen sets look like.

\begin{mylem}
\label{lemma1}
\textnormal{[CH].} For every $x$ in $\wstar$, the one-point compactification of a clopen non-compact subset of $\wstar \setminus \singleton{x}$ is homeomorphic to $\wstar$.
\end{mylem}

\begin{proof}
Let $A$ be a clopen non-compact subset of $\wstar \setminus \singleton{x}$. Taking $A \cup \singleton{x}$, a closed subset of $\wstar$, as representative of its one-point compactification, we see that it is a zero-dimensional compact $F$-space of weight $\cont$ without isolated points. 

For the $G_{\w_1}$-space property, suppose that $U \subset A \cup \singleton{x}$ is a non-empty $G_\delta$-set. If $U$ has empty intersection with $A$, then the singleton $U=\singleton{x}$ is a $G_\delta$-set, and hence has countable character  in the compact Hausdorff space $A \cup \singleton{x}$. It follows that there is a non-trivial sequence in $\wstar$ converging to $x$, a contradiction. Thus, $U$ intersects the open set $A$ and their intersection is a non-empty $G_\delta$-set of $\wstar$ with non-empty interior.

An application of Parovi\v{c}enko's theorem completes the proof.
\end{proof}

In absence of CH, the above proof still shows that the one-point compactification of any clopen non-compact subset of $\wstar \setminus \singleton{x}$ is a Parovi\v{c}enko space of weight $\cont$. 

It also follows that $\wstar$ contains $P$-points under CH, as the next lemma shows. Even more, we have another proof of \cite[2.3]{Three} that under CH for every point $x$ of $\wstar$ there is a clopen copy of $\wstar$ contained in $\wstar$ such that $x$ is a $P$-point with respect to that copy.

\begin{mycor}
\label{mycor1}
\textnormal{[CH].} For every point $x$ of $\wstar$, at least one of its wings together with $x$ itself is a copy of $\wstar$ such that $x$ is a $P$-point with respect to that copy. 
\end{mycor}

\begin{proof}
This follows from the last lemma together with the observation that if $x$ was a non-$P$-point with respect to both of its wings, it would be in the closure of two disjoint open $F_\sigma$-sets, contradicting the $F$-space property of $\wstar$.
\end{proof}

\section{\texorpdfstring{\bf Butterflies in $\mathbf{S_\kappa \setminus \singleton{x}}$}{Butterflies in S-kappa}}
\label{section4}

In this section we generalise results from the previous section about $\wstar = S_{\w_1}$ to general $S_\kappa$, assuming $\kappa=\kappa^{<\kappa}$ throughout. The reader is encouraged to check that the Butterfly Lemma and its immediate consequences carry over nicely to $S_\kappa$.

\begin{mylem}[Butterfly Lemma]
\label{butterfly2}
Assume $\kappa=\kappa^{<\kappa}$. Every point of $S_\kappa$ is a strong butterfly point. \qed
\end{mylem}

\begin{mycor}
\label{nonStone2}
Assume $\kappa=\kappa^{<\kappa}$. For every point $x$ of $S_\kappa$ the subspace $S_\kappa \setminus \singleton{x}$ is not $\Cstar$-embedded in $S_\kappa$. \qed
\end{mycor}

\begin{mycor}
\label{carclopens2}
Assume $\kappa=\kappa^{<\kappa}$. For every point $x$ in $S_\kappa$ we have $\cardinality{\clopens{S_\kappa \setminus \singleton{x}}} = 2^{\kappa}$. \qed
\end{mycor} 

The remaining part of this section is devoted to the proof of the following generalisation of Lemma \ref{lemma1}. 

\begin{mylem}
\label{superlemma}
Assume $\kappa = \kappa^{<\kappa}$ and let $x$ any point in $S_\kappa$. The one-point compactification of a clopen non-compact subset of $S_\kappa \setminus \singleton{x}$ is homeomorphic to $S_\kappa$.
\end{mylem}

The proof, however, is more delicate than in the case of $\wstar$. The challenge lies in the fact that Lemma \ref{lemma1} used as corner stones two facts about $F$-spaces which do not carry through to general $F_\kappa$-spaces: In normal spaces, the $F$-space property is closed-hereditary and every infinite closed subset of $S_{\w_1}$ has the same cardinality as $S_{\w_1}$. Both assertions do not hold for $F_\kappa$-spaces since $S_\kappa$ contains, being an infinite compact $F$-space, closed copies of $\beta \w$.

The following lemma is crucial in circumventing these obstacles. 

\begin{mylem}
\label{nicelemma}
Assume $\kappa = \kappa^{<\kappa}$ and let $x$ any point in $S_\kappa$. A clopen, non-compact subset of $S_\kappa \setminus \singleton{x}$ is of $S_\kappa$-type $\kappa$. 
\end{mylem}

\begin{proof}
Suppose for a contradiction that there exists a clopen, non-compact subset $A$ of $S_\kappa \setminus \singleton{x}$ of $S_\kappa$-type $\tau < \kappa$. Find a representation 
$$A=\bigcup_{\alpha < \tau} A_\alpha$$ where all $A_\alpha$ are clopen subsets of $S_\kappa$. We first claim that there is a collection $\set{V_\alpha}:{\alpha<\tau}$ of pairwise disjoint clopen sets of $S_\kappa$ such that $V_\alpha \subset A \setminus \bigcup_{\beta < \alpha}A_\beta$ for all $\alpha < \tau$. 

We proceed by transfinite recursion. Choose a clopen subset $V_0$ in the non-empty open set $A \setminus A_0$. Now consider $\alpha < \tau$ and suppose that $V_\beta$ have been defined for all $\beta < \alpha$. By minimality of $\tau$, the set $U_\alpha= \bigcup_{\beta < \alpha} A_\beta \cup V_\beta$ is a proper subset of $A$, from which it follows by $(5)$b of Section \ref{section2} that $U_\alpha$ is not dense in $A$. Thus, there is a non-empty clopen set $V_\alpha$ in the interior of $A \setminus U_\alpha$. This completes the recursion and proves the claim.

Now, let $f$ and $g$ be disjoint cofinal subsets of $\tau$. We define disjoint open sets 
$$V_f=\bigcup_{\alpha \in f} V_{\alpha} \quad \text{and} \quad V_g=\bigcup_{\alpha \in g} V_{\alpha}$$ of type at most $\tau$ and claim that both sets limit onto $x$, contradicting the $F_\kappa$-space property of $S_\kappa$. Suppose the claim was false. Then $\closure{V_f}$ is a subset of $A=\bigcup_{\alpha < \tau} A_\alpha$. By compactness, there is a finite set $F \subset \tau$ such that $\closure{V_f} \subset \bigcup_{\beta \in F} A_\beta$. But there are sets $V_\alpha$ with arbitrarily large index contributing to $V_f$, a contradiction.
\end{proof}

\begin{proof}[Proof of Lemma \ref{superlemma}]
Let $A$ be a clopen non-compact subset of $S_\kappa \setminus \singleton{x}$, and denote by $X$ the closure of $A$ in $S_\kappa$, i.e.\ $X=A \cup \singleton{x} \subset S_\kappa$. Then $X$ is a compact zero-dimensional space of weight $\kappa$ without isolated points. We check for the remaining $\kappa$-Parovi\v{c}enko properties.

To show that $X$ has the $F_\kappa$-space property, let $U$ and $V$ be disjoint open sets of $X$ of type less than $\kappa$. By normality, it suffices to show that $U$ and $V$ have disjoint closures in $X$. First suppose that $x$ belongs to $U \cup V$. Assume $x \in U$, so that $x$ does not belong to the closure of $V$. The sets $U\cap A$ and $V\cap A$ are of $A$-type less than $\kappa$. And since $A$ is an $F_\kappa$-space by \cite[14.1]{Ultrafilters}, they have disjoint closures in $A$, and therefore in $X$. Next, suppose that $x$ does not belong to $U \cup V$. Then $U$ and $V$ are subsets of $A$, and consequently of $S_\kappa$-type less than $\kappa$. Thus, $U$ and $V$ have disjoint closures in $S_\kappa$, and hence in $X$. This establishes that $X$ is an $F_\kappa$-space.

To show that $X$ has the $G_\kappa$-space property, suppose that $U=\bigcap_{\alpha < \lambda} U_\alpha$ is a non-empty set for $\lambda < \kappa$ where all $U_\alpha$ are clopen subsets of $X=A \cup \singleton{x}$. If $U$ has empty intersection with $A$, then all $X \setminus U_\alpha$ are clopen subsets of $S_\kappa$. It follows that $A=\bigcup_{\alpha < \lambda} X \setminus U_\alpha$ is a clopen non-compact subspace $S_\kappa \setminus \singleton{x}$ of type less than $\kappa$, contradicting Lemma \ref{nicelemma}. Thus, $U$ intersects $A$, and their intersection has non-empty interior in $S_\kappa$. 
\end{proof}

Following the proof of Corollary \ref{mycor1}, we see that for every point $x$ of $S_\kappa$ there is a clopen copy of $S_\kappa$ contained in $S_\kappa$ such that $x$ is a $P_\kappa$-point with respect to that copy. In particular, $S_\kappa$ contains $P_\kappa$-points.
 
\begin{mycor}
\label{mycor12}
Assume $\kappa = \kappa^{<\kappa}$. For every point $x$ of $S_\kappa$ and for every butterfly around $x$, one of its wings together with $x$ itself is a copy of $S_\kappa$ such that $x$ is a $P_\kappa$-point with respect to that copy. \qed
\end{mycor}

\section{\texorpdfstring{\bf The Stone-\v{C}ech compactification of $\mathbf{S_\kappa \setminus \singleton{x}}$}{The Stone-Cech Compactification}}
\label{section6}

In this section we show that $\beta (S_\kappa \setminus \singleton{x})$ is a $\kappa$-Parovi\v{c}enko space of large weight. The result is best possible in the sense that $\beta (S_\kappa \setminus \singleton{x})$ cannot be a $\kappa^+$-Parovi\v{c}enko space, for points in $S_\kappa \setminus \singleton{x}$ continue to have character $\kappa$.

We first prove a theorem about $F_\kappa$-spaces which generalises the well-known corresponding theorem for $F$-spaces (compare ($1$) and ($1'$) in Section \ref{section2}). It is interesting to note that the $F$-space property is even hereditary with respect to $C^*$-embedded subspaces, but the same is not true for $F_\kappa$-spaces---see the remarks after Lemma \ref{superlemma}.

\begin{mythm}
\label{FKappaSpaces}
A strongly zero-dimensional space is an $F_\kappa$-space if and only if its Stone-\v{C}ech compactification is an $F_\kappa$-space.
\end{mythm}

\begin{proof} Suppose that $X$ is an $F_\kappa$-space, and $U$ an open set of $\beta X$ of type less than $\kappa$. To show that $U$ is $C^*$-embedded in $\beta X$, fix a continuous $[0,1]$-valued function $f$ defined on $U$. The set $U\cap X$ is of type less than $\kappa$ in $X$. By assumption, $f|_{U\cap X}$ can be extended to a continuous $[0,1]$-valued function $F$ on $X$ which again can be extended to a continuous $[0,1]$-valued function $\beta F$ on $\beta X$. Now $\beta F|_U = f$, as both functions agree on the dense subset $U\cap X$.  

For the converse, assume that $\beta X$ is an $F_\kappa$-space. We show by induction on $\lambda$ that $X$ is an $F_\lambda$-space for all $\lambda \leq \kappa$. For $\lambda = \w$ there is nothing to prove. So let $\lambda \leq \kappa$ be uncountable and assume that $X$ is an $F_\mu$-space for all $\mu < \lambda$. Let $U$ be an open set of $X$-type $\tau < \lambda$. We aim to show that $U$ is $C^*$-embedded in $X$. 

There are $X$-clopen sets $U_\alpha$ such that $U=\bigcup_{\alpha < \tau} U_\alpha$. Write $V_\beta = \bigcup_{\alpha<\beta}U_\alpha$ and $W_\beta = \bigcup_{\alpha < \beta} \closure{U}_\alpha$ where the closure is taken in $\beta X$. Note that all $\closure{U}_\alpha$ are clopen subsets of $\beta X$ and that $V_\beta = W_\beta \cap X$. Let $f$ be a continuous $[0,1]$-valued map on $V_\tau=U$. For each $\beta < \tau$, the set $V_\beta$ is $C^*$-embedded in $X$ by induction hypothesis, and hence, $f|_{V_\beta}$ extends to $X$ and then to $\beta X$. Let $f_\beta$ be the restriction of this extension to $W_\beta$.

Since $V_\beta=W_\beta \cap X$ is dense in $W_\beta$ for all $\beta < \tau$, the function $f_\tau=\bigcup_{\beta < \tau} f_\beta$ is well-defined on $W_\tau$. And since every $\closure{U}_a$ is a clopen subset of $\beta X$, is is not hard to check that $f_\tau$ is continuous. Thus, $f_\tau$ extends from $W_\tau$ to $\beta X$ by the $F_\kappa$-space property. The restriction of this extension to $X$ is the required extension of $f$.
\end{proof}

\begin{mythm}
\label{main1}
Assume $\kappa = \kappa^{<\kappa}$. For every point $x$ of $S_\kappa$ the space $\beta (S_\kappa \setminus \singleton{x})$ is a $\kappa$-Parovi\v{c}enko space of weight $2^\kappa$. 
\end{mythm}

\begin{proof}
Let $X=S_\kappa \setminus \singleton{x}$. First we verify the $\kappa$-Parovi\v{c}enko properties. Clearly, $\beta X$ does not contain isolated points. By  \cite[14.1]{Ultrafilters}, every open subspace of $S_\kappa$ is a strongly zero-dimensional $F_\kappa$-space and hence so is $\beta X$ by Theorem \ref{FKappaSpaces}.

For the $G_\kappa$-space property, let $U=\bigcap_{\alpha < \lambda} U_\alpha$ be a non-empty intersection of clopen sets in $\beta X$ for some $\lambda < \kappa$. To prove that $U$ has non-empty interior it suffices to show that it intersects $X$.

Assume for a contradiction that $\lambda$ is minimal such that $U$ has empty intersection with $X$. Consider the sets $V_\beta = \bigcap_{\alpha < \beta} U_\alpha \cap X$. Then $\bigcap_{\beta < \lambda} V_\beta$ is empty, whereas, without loss of generality, $V_\beta \setminus V_{\beta +1}$ is non-empty for all $\beta < \lambda$. This last set is a non-empty intersection of less than $\kappa$-many open sets in $S_\kappa$ and therefore contains a compact clopen set $W_\beta$.

Let $f$ and $g$ be disjoint cofinal subsets of $\lambda$. We define disjoint open sets
$$W_f=\bigcup_{\alpha \in f} W_{\alpha} \quad \text{and} \quad W_g=\bigcup_{\alpha \in g} W_{\alpha}$$ of type at most $\lambda<\kappa$ and claim that in $S_\kappa$ both sets limit onto $x$, contradicting the $F_\kappa$-space property. Suppose the claim was false, e.g.\ that $W_f$ does not limit onto $x$. Then the closure of $W_f$ in $S_\kappa$, a compact set, would be contained in $X=\bigcup_{\beta < \lambda} X \setminus V_\beta$. But by construction of $W_f$, this open cover of $\closure{W_f}$ does not have a finite subcover, giving the desired contradiction.

It remains to calculate the weight of $\beta (S_\kappa \setminus \singleton{x})$. By Theorem \cite[2.21]{Ultrafilters} and Lemmas \cite[2.23(b) \& 2.24]{Ultrafilters}, the weight of $\beta X$ for a strongly zero-dimensional space $X$ is equal to the cardinality of $\clopens{X}$. Now apply Corollary \ref{carclopens2}.
\end{proof}

It is an interesting question whether it is a ZFC theorem that the Stone-\v{C}ech compactification of $\wstar \setminus \singleton{x}$ is a Parovi\v{c}enko space. A.\,Dow showed that the assertion that all open subspaces of $\wstar$ are strongly zero-dimensional $F$-spaces is equivalent to CH \cite{Dow2}. However, we are only interested in open subspaces of the form $\wstar \setminus \singleton{x}$. And these may be strongly zero-dimensional $F$-spaces even under the negation of CH, for example when $\wstar \setminus \singleton{x}$ is $\Cstar$-embedded in $\wstar$.

The result about the weight also follows from the fact that the remainder of $S_\kappa \setminus \singleton{x}$ has weight $2^\kappa$, see Theorem \ref{cellularity}.

In the case of $\kappa=\w_1$, a tempting proof of Theorem \ref{main1} can be obtained by observing that $\wstar \setminus \singleton{x}$ is pseudocompact (as it cannot contain a closed copy of $\w$), implying that every $G_\delta$-set of $\beta (\wstar \setminus \singleton{x})$ intersects $\wstar \setminus \singleton{x}$. But this approach does not seem to generalise to $S_\kappa \setminus \singleton{x}$ without  extra effort. At the same time, both approaches are somehow intertwined: the proof of Theorem \ref{main1} shows that $S_\kappa \setminus \singleton{x}$ is $\alpha$-pseudocompact for all $\alpha < \kappa$ \cite[2.2]{alphapseudo}. For more on $\alpha$-pseudocompactness see \cite{alphapseudo}.


\section{\texorpdfstring{\bf The structure of $\mathbf{(S_\kappa \setminus \singleton{x})^*}$}{The Structure of Remainders}}
\label{section7}

The previous section has shown that $\beta (S_\kappa \setminus \singleton{x})$ is a $\kappa$-Parovi\v{c}enko space of quite large weight. Thus, these spaces are too large for the $\kappa$-Parovi\v{c}enko properties to provide meaningful topological restrictions on the variety of potential spaces of that size. 

For a better understanding we therefore turn to an investigation of the remainders of these spaces. The main result of this section is that every space of the form $(S_\kappa \setminus \singleton{x})^*$ is a $\kappa^+$-Parovi\v{c}enko space of weight $2^\kappa$, regardless of the choice of $x$. It follows that under $2^\kappa=\kappa^+$, all such remainders are homeomorphic to $S_{\kappa^+}$. In particular, it is consistent with CH that all remainders of spaces of the form $\wstar \setminus \singleton{x}$ are homeomorphic. 

Note that by Theorem \ref{main1}, $(S_\kappa \setminus \singleton{x})^*$ is a compact zero-dimensional space. The next lemma is the first step for establishing the remaining Parovi\v{c}enko properties.

\begin{mylem}
\label{atomless2}
Assume $\kappa=\kappa^{<\kappa}$. For every $x \in S_\kappa$ the space $(S_\kappa \setminus \singleton{x})^*$ has no isolated points.
\end{mylem}

\begin{proof}
Suppose that $z \in (S_\kappa \setminus \singleton{x})^*$ is isolated. Then there is a clopen non-compact subset $A$ of $S_\kappa \setminus \singleton{x}$ such that $A \cup \singleton{z}$ is compact and $A$ is $C^*$-embedded in $A \cup \singleton{z}$. By Lemma \ref{superlemma}, the set $A$ is homeomorphic to $S_\kappa \setminus \singleton{y}$ for some $y \in S_\kappa$. However, this space does not have a one-point Stone-\v{C}ech compactification by Corollary \ref{nonStone2}, a contradiction.
\end{proof}

Our next observation is that for a $P_\kappa$-point $p$, the space $S_\kappa \setminus \singleton{p}$ can be written as an \emph{increasing} union of $\kappa$-many compact clopen sets, each homeomorphic to $S_\kappa$. And remainders of increasing unions of $\kappa$-Parovi\v{c}enko spaces are well understood: the next theorem follows from a result by A.\,Dow from 1985 \cite[2.2]{Dow}.

\begin{mythm}[Dow]
\label{Dowsresult}
Assume $\kappa = \kappa^{<\kappa}$. For a $P_\kappa$-point $p$ of $S_\kappa$ the space $(S_\kappa \setminus \singleton{p})^*$ is a $\kappa^+$-Parovi\v{c}enko space. \qed
\end{mythm}

Our key result is the following strengthening of Theorem \ref{Dowsresult}. 

\begin{mythm}
\label{hui}
Assume $\kappa = \kappa^{<\kappa}$. For every point $x$ of $S_\kappa$ the space $(S_\kappa \setminus \singleton{x})^*$ is a $\kappa^+$-Parovi\v{c}enko space of weight $2^\kappa$.
\end{mythm}

We remark that our proof of Theorem \ref{hui} makes essential use of Dow's original theorem, and does not handle all cases simultaneously. Before presenting the proof, we discuss some interesting corollaries.

\begin{mycor}
Assume $\kappa = \kappa^{<\kappa}$. The Stone-\v{C}ech compactifications $\beta(S_\kappa \setminus \singleton{x})$ and $\beta(S_\kappa \setminus \singleton{y})$ are homeomorphic if and only if $S_\kappa \setminus \singleton{x}$ and $S_\kappa \setminus \singleton{y}$ are homeomorphic.
\end{mycor}

\begin{proof}
By Theorem \ref{hui}, points in the ground space can be distinguished from points in the remainder of $\beta(S_\kappa \setminus \singleton{x})$ by their character. Thus, any homeomorphism between $\beta(S_\kappa \setminus \singleton{x})$ and $\beta(S_\kappa \setminus \singleton{y})$ restricts to a homeomorphism between $S_\kappa \setminus \singleton{x}$ and $S_\kappa \setminus \singleton{y}$.
\end{proof}

\begin{mycor}
\label{result}
Assume $\kappa = \kappa^{<\kappa}$ and $2^\kappa=\kappa^+$. For every point $x$ the space $(S_\kappa \setminus \singleton{x})^*$ is homeomorphic to $S_{\kappa^+}$.
\end{mycor}

\begin{proof}
Write $\lambda=\kappa^+$. The condition $2^\kappa = \kappa^+$ implies $\lambda = \lambda^{<\lambda}$. Thus, for every $x$ in $S_\kappa$, the space $(S_\kappa \setminus \singleton{x})^*$ is a $\lambda$-Parovi\v{c}enko space of weight $\lambda = \lambda^{<\lambda}$ by Theorem \ref{hui}, and hence, by Negrepontis' characterisation, homeomorphic to $S_\lambda$.
\end{proof}

\begin{mycor}
\label{chchchch}
Under the cardinal assumption $2^\cont=\w_2$, the remainders of $\wstar \setminus \singleton{x}$ and $\wstar \setminus \singleton{y}$ are homeomorphic for all points $x$ and $y$ of $\wstar$. \qed
\end{mycor}

This is especially interesting when compared to the fact that there are $2^\cont$ non-homeomorphic subspaces of the form $\wstar \setminus \singleton{x}$, an observation that follows easily from Frol\'ik's result \cite{Frolik} that there are $2^\cont$ orbits in $\wstar$ under its autohomeomorphism group, i.e.\ that $\wstar$ is badly non-homogenous. 

The remaining part of this section is devoted to the proof of Theorem \ref{hui}. For this, we first need a lemma about separation of disjoint open sets of small type in $S_\kappa \setminus \singleton{x}$. Note that a clopen non-compact set of $S_\kappa \setminus \singleton{x}$ is of $S_\kappa$-type $\kappa$ by Lemma \ref{nicelemma}, but of $(S_\kappa \setminus \singleton{x})$-type $1$.

\begin{mylem}
\label{moretechnicalbits}
Assume $\kappa = \kappa^{<\kappa}$. For any two disjoint open sets $V$ and $W$ of $(S_\kappa \setminus \singleton{x})$-type less than $\kappa$, there is a clopen set $A$ of $S_\kappa \setminus \singleton{x}$ such that $V \subset A$ and $W \cap A = \emptyset$.  
\end{mylem}

\begin{proof}
The space $S_\kappa \setminus \singleton{x}$ is an $F_\kappa$-space by \cite[14.1]{Ultrafilters}, so $V$ and $W$ have disjoint closures in $S_\kappa \setminus \singleton{x}$. If $V$ and $W$ have disjoint closures in $S_\kappa$, by compactness there is a clopen set $A \subset S_\kappa$ separating $V$ from $W$. Clearly, the intersection of $A$ with $S_\kappa \setminus \singleton{x}$ is as required. 

So assume that both $V$ and $W$ limit onto $x$. We build a butterfly such that its wings separate $V$ and $W$. Let $\set{U_\alpha}:{\alpha < \kappa}$ be a clopen neighbourhood base for $x$ in $S_\kappa$. By transfinite recursion we define $S_\kappa$-clopen sets $A_\alpha$ and $B_\alpha$ for $\alpha < \kappa$ not containing $x$ such that all all $(A_\beta,B_\gamma)$-, $(A_\beta,W)$- and $(V,B_\gamma)$-pairs are disjoint and $A_\beta \cup B_\beta$ covers $S_\kappa \setminus U_\beta$ for all $\beta,\gamma < \kappa$.

Once the construction is completed, we define disjoint open sets $A = \Union_{\alpha < \kappa} A_\alpha$ and $B = \Union_{\alpha < \kappa} B_\alpha$. Their union covers of all of $S_\kappa \setminus \singleton{x}$ and $V \subset A$ and $W \subset B$ as required.

It remains to complete the recursive construction. Let $\alpha < \kappa$ and assume that $A_\beta$ and $B_\beta$ have been defined for all  ordinals $\beta < \alpha$ satisfying the inductive requirements. Since $V$ is of $(S_\kappa \setminus \singleton{x})$-type less than $\kappa$, it follows easily that $V \setminus U_\alpha$ is of $S_\kappa$-type less than $\kappa$. So the sets 
$$(V \cup \bigcup_{\beta< \alpha} A_\beta) \setminus U_\alpha \quad \text{and} \quad (W \cup \bigcup_{\beta< \alpha} B_\beta) \setminus U_\alpha$$ are disjoint open sets of $S_\kappa$-type less than $\kappa$, and by the $F_\kappa$-space property there is a clopen partition $(C,D)$ of $S_\kappa$ separating them. We put $A_\alpha = C\setminus U_\alpha$ and $B_\alpha = D\setminus U_\alpha$, preserving the inductive assumptions.  
\end{proof}

\begin{proof}[Proof of Theorem \ref{hui}]
Because of Theorem \ref{Dowsresult}, it suffices to prove the theorem for non-$P_\kappa$-points $x$. So let us fix a non-$P_\kappa$-point $x$ of $S_\kappa$ and an open subset $V \subset S_\kappa$ of type less than $\kappa$ that contains $x$ in its boundary. By the $F_\kappa$-space property, $V$ is $\Cstar$-embedded in $S_\kappa$. In particular, if we write $X=S_\kappa \setminus \singleton{x}$ then the closure of $V$ in $X$ has a one-point Stone-\v{C}ech compactification. This means the set $V$ limits onto precisely one point in the remainder of $X$. For the remaining parts of this proof, we denote this unique point in $\closureIn{V}{\beta X} \setminus X$ by $\star$. 

\begin{myclm}
\label{claim1}
For every clopen non-empty set $C$ of $X^*$ not containing $\star$ there is a clopen non-compact set $D \subset X$ which misses $V$ such that $D^*=\closure{D} \setminus D = C$. Moreover, every such $D$ is homeomorphic to $S_\kappa \setminus \singleton{p}$ for a $P_\kappa$-point $p$.
\end{myclm}

To see that the claim holds, let $C$ be a clopen subset of $X^*$ not containing $\star$ and find a clopen non-compact subset $E$ of $X$ with $E^*=C$. There is a clopen neighbourhood $U$ of $x$ in $S_\kappa$ such that $U \cap (E \cap V) = \emptyset$: otherwise, the closure of $E \cap V$ in $\beta X$ would grow into the remainder. But $$\closureIn{E \cap V}{\beta X} \setminus X \subset  \closureIn{E}{\beta X} \setminus X \cap \closureIn{V}{\beta X} \setminus X = C \cap \singleton{\star} = \emptyset,$$
a contradiction. Hence, for some suitable $U$, the clopen non-compact set $D=E\cap U$ of $S_\kappa$ does not intersect $V$. And since the symmetric difference of $D$ and $E$ is compact, it follows from \cite[2.6d]{Ultrafilters} that $D^* =E^* =C$, as claimed.

To see that $D$ is homeomorphic to $S_\kappa \setminus \singleton{p}$ for a $P_\kappa$-point $p$, note that $D$ and $X \setminus D$ form a butterfly around $x$ in $S_\kappa$ such that $V$ witnesses that $x$ is not a $P_\kappa$-point of $S_\kappa \setminus D$. Hence, $D$ is homeomorphic to $S_\kappa \setminus \singleton{p}$ for a $P_\kappa$-point $p$ by Corollary \ref{mycor12}, completing the proof of Claim \ref{claim1}. 

\begin{myclm}
\label{claim2}
Every compact clopen set of $X^*\setminus \singleton{\star}$ is a $\kappa^+$-Parovi\v{c}enko space.
\end{myclm}

By Claim \ref{claim1}, a compact clopen set of $X^*\setminus \singleton{\star}$ is of the form $(S_\kappa \setminus \singleton{p})^*$ for a $P_\kappa$-point $p$ of $S_\kappa$. Claim \ref{claim2} now follows from Theorem \ref{Dowsresult}.

\begin{myclm}
\label{claim3}
The point $\star$ is a $P_{\kappa^+}$-point of $X^*$.
\end{myclm}
 
To prove Claim \ref{claim3} we show that whenever $\set{C_\alpha}:{\alpha < \kappa}$ is a collection of $X^*$-clopen sets not containing $\star$, there is a clopen set $B^* \subset X^*$ not containing $\star$ such that $\bigcup_{\alpha < \kappa} C_\alpha \subset B^*$.

From Claim \ref{claim1}, we know that for every $C_\alpha$ there is a clopen non-compact subset $D_\alpha \subset X$ such that $D_\alpha \cap V= \emptyset$ and $D_\alpha^*=C_\alpha$. In $X$, we write $F \subset^* G$ (read: $F$ is \emph{almost contained} in $G$) if there is a clopen neighbourhood $U$ of $x$ in $S_\kappa$, such that $F \cap U \subset G$. Write $F=^*G$ if $F\subset^*G$ and $G \subset^*F$. We will use the well-known fact that $C_\alpha \subset C_\beta$ in $X^*$ if and only if $D_\alpha \subset^* D_\beta$ in $X$. 

Similar to the proof of Lemma \ref{moretechnicalbits}, our aim is to build a butterfly in $X$ with wings $A$ and $B$ such that $V \subset A$ and $D_\alpha \subset^* B$ for all $\alpha < \kappa$. Clearly then, $B^*$ is as required. 

To construct such a butterfly, fix a neighbourhood base $\set{U_\alpha}:{\alpha < \kappa}$ of $x$ in $S_\kappa$ consisting of clopen sets. By recursion we will define families of $S_\kappa$-clopen sets $\set{A_\alpha}:{\alpha < \kappa}$ and $\set{B_\alpha}:{\alpha < \kappa}$  and a third family $\set{E_\alpha}:{\alpha < \kappa}$ of $X$-clopen sets such that for all $\alpha, \beta < \kappa$
\begin{enumerate}
\item $A_\alpha \cap B_\beta = \emptyset$ and $A_\alpha \cup B_\alpha = S_\kappa \setminus U_\alpha$, 
\item $A_\alpha \cap E_\beta = \emptyset$ and $B_\alpha \cap V = \emptyset$,
\item $E_\alpha \subset D_\alpha$ and $E_\alpha=^*D_\alpha$.
\end{enumerate} 
Once the construction is completed, it follows from $(1)$ that $A=\bigcup_{\alpha<\kappa}A_\alpha$ and $B=\bigcup_{\alpha<\kappa}B_\alpha$ partition $X$ into disjoint open sets. Condition $(2)$ guarantees both $V \subset A$ and $E_\alpha \subset B$ and finally, condition $(3)$ gives $D_\alpha \subset^* B$ as desired. 

It remains to complete the recursive construction. Let $\alpha < \kappa$ and assume that $A_\beta$, $B_\beta$ and $E_\beta$ have been defined for all ordinals $\beta < \alpha$ satisfying the inductive assumptions. The set $S_\kappa \setminus \bigcup_{\beta < \alpha} A_\beta$ is an intersection of less than $\kappa$-many clopen sets in $S_\kappa$ containing $x$. In particular, it is a non-empty intersection of less than $\kappa$-many clopen sets in $D_\alpha \cup \singleton{x}$. But by Claim \ref{claim1}, the point $x$ is a $P_\kappa$-point with respect to $D_\alpha \cup \singleton{x}$, and hence there is a $D_\alpha \cup \singleton{x}$-clopen neighbourhood $E'_\alpha$ of $x$ in this space such that $E'_\alpha \subset S_\kappa \setminus \bigcup_{\beta < \alpha} A_\beta$. Now put $E_\alpha = E'_\alpha \setminus \singleton{x}$ and note that $E_\alpha=^*D_\alpha$.

By Lemma \ref{moretechnicalbits} there exist $X$-clopen sets $C$ and $D$ partitioning $X$ and containing the disjoint open sets 
$$V \cup \Union_{\beta < \alpha} A_\beta \quad \text{and} \quad \Union_{\beta < \alpha} B_\beta \cup \bigcup_{\beta \leq \alpha} E_\beta$$ respectively. By defining  $A_\alpha = C \setminus U_\alpha$ and $B_\alpha = D \setminus U_\alpha$ it is clear that $A_\alpha$, $B_\alpha$ and $E_\alpha$ satisfy the inductive assumptions $(1)$-$(3)$. The proof of Claim \ref{claim3} is complete.

\begin{myclm}
\label{claim4}
The space $X^*$ is a $\kappa^+$-Parovi\v{c}enko space.
\end{myclm}

For this, note that $X^*$ is a zero-dimensional compact space without isolated points by Lemma \ref{atomless2}. Thus it only remains to check for the $F_{\kappa^+}$- and the $G_{\kappa^+}$-space property. So suppose one of these conditions fails. This is witnessed by some point $x$. By Claim \ref{claim2}, $x$ must be $\star$, but this is a contradiction as $\star$ is a $P_{\kappa^+}$-point by Claim \ref{claim3}. This proves Claim \ref{claim4}.

To complete the proof, it remains to establish the weight of $X^*$. But every $\kappa^+$-Parovi\v{c}enko space has weight at least $2^\kappa$, which can either be shown directly by embedding a binary tree of clopen sets of height $\kappa+1$, or, in the case of $X^*$, be concluded from Theorem \ref{cellularity}. 
\end{proof}

\section{\texorpdfstring{\bf Cardinal invariants of $\mathbf{\left(S_\kappa \setminus \singleton{x}\right)^*}$}{Cardinal Invariants of Remainders}}
\label{section8}

The results from the previous section have shown that under $2^{\kappa}=\kappa^+$, all remainders of the form $(S_\kappa \setminus \singleton{x})^*$ are homeomorphic to $S_{\kappa^+}$, which settles all further topological questions regarding cardinality, weight and cellularity of these spaces. However, we do not yet know what happens in absence of $2^{\kappa}=\kappa^+$. In this section we therefore investigate cardinal invariants of $(S_\kappa \setminus \singleton{x})^*$ without additional set-theoretic assumptions beyond its existence, i.e.\ $\kappa=\kappa^{<\kappa}$.

The next lemma says that for questions such as size, weight and cellularity of the remainder of $S_\kappa \setminus \singleton{x}$, it is enough to focus on the remainder of $S_\kappa \setminus \singleton{p}$ for a $P_\kappa$-point $p$.
 
\begin{mylem}
\label{wlog}
Assume $\kappa=\kappa^{<\kappa}$. For all $x \in S_\kappa$ the space $(S_\kappa \setminus \singleton{x})^*$ contains a clopen copy of $(S_\kappa \setminus \singleton{p})^*$ for a $P_\kappa$-point $p$.
\end{mylem}

\begin{proof}
By Corollary \ref{mycor12}, the space $S_\kappa \setminus \singleton{x}$ contains a clopen copy of $S_\kappa \setminus \singleton{p}$ for a $P_\kappa$-point $p$.
\end{proof}

Every point of a $\kappa$-Parovi\v{c}enko space has character at least $\kappa$ and hence the whole space has, by \cite[3.12.11]{Eng}, cardinality at least $2^\kappa$. Thus, Lemma \ref{wlog} and Theorem \ref{Dowsresult} give us the inequality $$2^{\kappa^+} \leq \cardinality{(S_\kappa \setminus \singleton{p})^*} \leq \cardinality{(S_\kappa \setminus \singleton{x})^*} \leq 2^{2^\kappa}.$$
In particular, we see once again that under $2^\kappa = \kappa^+$, the cardinality of $(S_\kappa \setminus \singleton{x})^*$ is clear. In the remaining part of this paper we show without additional set-theoretic assumptions that $(S_\kappa \setminus \singleton{p})^*$---and hence $(S_\kappa \setminus \singleton{x})^*$ for all $x$---has maximal cardinality. We also provide a topological construction showing that $(S_\kappa \setminus \singleton{x})^*$ has maximal cellularity $2^\kappa$. Although this last result can also be concluded directly from Theorem \ref{Dowsresult}, it is interesting to see the topological similarities to the relation between MAD families on $\w$ and the cellularity of $\wstar$. 

Before beginning with our preparations we remark that Comfort and Negrepontis observed in \cite[2.4]{Three} that under CH, the space $(\wstar \setminus \singleton{p})^*$ has cardinality $2^{2^{\w_1}}$. Their method can also be used to determine the cellularity of $(\wstar \setminus \singleton{p})^*$. However, the proof does not generalise to $S_\kappa$ for $\kappa > \w_1$. Indeed, it seems to be a recurring nuisance that elegant proofs about $\wstar$ that at their heart invoke the Stone-\v{C}ech properties of $\wstar = \beta \w \setminus \w$ do not carry over to general $S_\kappa$. In the following we present an approach that is solely based on the Parovi\v{c}enko properties.

The next lemma guarantees the existence of a \emph{monotone cut operator} for all pairs of disjoint open sets of type less than $\kappa$ in $S_\kappa$, in the spirit of a separation operator for monotone normality. We make this definition precise. For an ordered pair $\pair{A,B}$ of disjoint open sets of $S_\kappa$ of type less than $\kappa$, we define a \emph{cut} between them to be a clopen set $\script{C}_{\pair{A,B}}$ such that 
$$A \subset \script{C}_{\pair{A,B}} \subset S_\kappa \setminus B.$$
Cuts in $S_\kappa$ exist by the $F_\kappa$-space property. We write $\pair{A,B} \leq \pair{A',B'}$ if $A \subset A'$ and $B \supset B'$. A cut operator $\script{C}$ is called \emph{monotone} if $\pair{A,B} \leq \pair{A',B'}$ implies $\script{C}_{\pair{A,B}} \subset  \script{C}_{\pair{A',B'}}$. The cut operator is called \emph{symmetric} if $\script{C}_{\pair{A,B}} = S_\kappa \setminus \script{C}_{\pair{B,A}}$.

We need the following strengthening of the concept of a monotone cut operator. Let $\set{U_\alpha}:{\alpha < \kappa}$ be a decreasing neighbourhood base of clopen sets of a $P_\kappa$-point $p$ such that $U_0=S_\kappa$. Call two subsets $A$ and $B$ of $S_\kappa$ $\gamma$-\emph{equivalent} if $A \cap U_\gamma = B \cap U_\gamma$ for some $\gamma < \kappa$. Also, call $A$ a $\gamma$-\emph{subset} of $B$, and write $A \subset_\gamma B$, if $A \cap U_\gamma \subset B \cap U_\gamma$. We extend this idea to capture ``local monotonicity'' and write $\pair{A,B} \leq_\gamma \pair{A',B'}$ if $A \subset_\gamma A'$ and $B \supset_\gamma B'$. Note that for $\gamma\leq\delta<\kappa$, if $\pair{A,B} \leq_\gamma \pair{A',B'}$ holds then so does $\pair{A,B} \leq_\delta \pair{A',B'}$. 

A cut operator $\script{C}$ with the property such that for all $\gamma$, $\pair{A,B} \leq_{\gamma} \pair{A',B'}$ implies $\script{C}_{\pair{A,B}} \subset_\gamma \script{C}_{\pair{A',B'}}$ will be called a \emph{strong monotone cut operator with respect to $\Set{U_\alpha}$}. Every strong monotone cut operator is also monotone.

\begin{mylem}
\label{monotone2}
Assume $\kappa = \kappa^{<\kappa}$. Let $\script{F}$ be the collection of all ordered pairs of disjoint open sets of $S_\kappa$ of type less than $\kappa$. For every decreasing neighbourhood base of clopen sets $\set{U_\alpha}:{\alpha < \kappa}$ of a $P_\kappa$-point in $S_\kappa$ with $U_0=S_\kappa$, there exists a symmetric strong monotone cut operator operator $\script{C} \colon \script{F} \to \clopens{S_\kappa}$ with respect to $\Set{U_\alpha}$. 
\end{mylem}

\begin{proof}
By $\kappa = \kappa^{<\kappa}$ we may list $\script{F}=\set{\pair{A_\alpha,B_\alpha}}:{\alpha < \kappa}$, such that permuted pairs are next to each other. Let $\set{U_\alpha}:{\alpha < \kappa}$ be a decreasing neighbourhood base of a $P_\kappa$-point $p$ consisting of clopen sets such that $U_0=S_\kappa$. In addition we define $U_{\kappa}=\emptyset$, obtaining the technical advantage that for all pairs of sets, one is a $\gamma$-subset of the other for some $\gamma \leq \kappa$.

We define an operator $\script{C} \colon \script{F} \to \clopens{S_\kappa}$ that satisfies for all ordinals $\beta, \delta < \kappa$:
\begin{eqnarray}
&(Cut)& A_\beta \subset \script{C}_{\pair{A_\beta,B_\beta}} \subset S_\kappa \setminus B_\beta,  \nonumber \\
&(Sym)& \script{C}_{\pair{A_\beta,B_\beta}}= S_\kappa \setminus\script{C}_{\pair{B_\beta,A_\beta}} \; \text{ and } \nonumber \\
&(Mon)&  \forall \gamma < \kappa \; (\pair{A_\delta,B_\delta} \leq_{\gamma} \pair{A_\beta,B_\beta} \; \Rightarrow \; \script{C}_{\pair{A_\delta,B_\delta}} \subset_\gamma \script{C}_{\pair{A_\beta,B_\beta}}). \nonumber
\end{eqnarray}

We proceed by transfinite recursion. Let $\alpha < \kappa$ and suppose we have defined cuts $\script{C}_{\pair{A_\beta,B_\beta}}$ for all $\beta < \alpha$ satisfying the inductive assumptions for all $\beta,\delta < \alpha$.

Consider $\pair{A_\alpha,B_\alpha}$. If $\alpha$ is a successor and $\pair{A_\alpha,B_\alpha} = \pair{B_{\alpha-1},A_{\alpha-1}}$ we define $\script{C}_{\pair{A_\alpha,B_\alpha}}=S_\kappa \setminus \script{C}_{ \pair{A_{\alpha-1},B_{\alpha-1}}}$. This assignment takes care of $(Sym)$ and a straightforward calculation shows that also $(Cut)$ and $(Mon)$ are satisfied. 

Otherwise,  for all $\beta < \alpha$ we let $\gamma^\downarrow_\beta$ and $\gamma^\uparrow_\beta$ be the least ordinals such that $\pair{A_\beta,B_\beta} \leq_{\gamma^\downarrow_\beta} \pair{A_\alpha,B_\alpha}$ and $ \pair{A_\alpha,B_\alpha} \leq_{\gamma^\uparrow_\beta} \pair{A_\beta,B_\beta} $. This is well-defined, as these relations are satisfied for at least $\kappa$. Let
$$
\script{C}^\downarrow_\alpha = \set{U_{\gamma^\downarrow_\beta} \cap \script{C}_{\pair{A_\beta,B_\beta}}}:{\beta < \alpha}  \;\text{ and } \; \script{C}^\uparrow_\alpha = \set{U_{\gamma^\uparrow_\beta} \setminus \script{C}_{\pair{A_\beta,B_\beta}}}:{ \beta < \alpha}.
$$

The idea is that these sets contain all parts of the previously defined cuts we have to be aware of in order to make our operator respect $(Mon)$. Both sets have cardinality less than $\kappa$ and consist of clopen sets. 

We claim that the sets $A_\alpha \cup (\bigcup \script{C}^\downarrow_\alpha)$ and $B_\alpha \cup(\bigcup \script{C}^\uparrow_\alpha)$ are disjoint open sets of type less than $\kappa$. They are clearly open and of type less than $\kappa$. 

We demonstrate only that $\bigcup \script{C}^\downarrow_\alpha$ and $\bigcup \script{C}^\uparrow_\alpha$ are disjoint, since the other cases are similar. For this we show that for any $\beta,\delta < \alpha$, the sets  $U_{\gamma_\beta^\downarrow} \cap \script{C}_{\pair{A_\beta,B_\beta}}  \in \script{C}^\downarrow_\alpha$ and $U_{\gamma_\delta^\uparrow} \setminus \script{C}_{\pair{A_\delta,B_\delta}} \in \script{C}^\uparrow_\alpha$ have empty intersection. By construction we have 
	$$\pair{A_\beta,B_\beta} \leq_{\gamma_\beta^\downarrow} \pair{A_\alpha,B_\alpha} \; \text{ and } \; \pair{A_\alpha,B_\alpha} \leq_{\gamma_\delta^\uparrow} \pair{A_\delta,B_\delta}.$$ 
	With $\gamma$ denoting the larger of $\gamma_\beta^\downarrow$ and $\gamma_\delta^\uparrow$ we may apply condition $(Mon)$ to $\pair{A_\beta,B_\beta} \leq_\gamma \pair{A_\delta,B_\delta}$ and obtain $\script{C}_{\pair{A_\beta,B_\beta}} \subset_\gamma \script{C}_{\pair{A_\delta,B_\delta}}$. 
		In particular, the sets $U_{\gamma} \cap \script{C}_{\pair{A_\beta,B_\beta}}$ and $U_\gamma \setminus \script{C}_{\pair{A_\delta,B_\delta}}$ have empty intersection, and since $U_\gamma = U_{\gamma_\beta^\downarrow} \cap U_{\gamma_\delta^\uparrow}$, the result follows. 
		
Now, since $A_\alpha \cup (\bigcup \script{C}^\downarrow_\alpha)$ and $B_\alpha \cup(\bigcup \script{C}^\uparrow_\alpha)$ are disjoint open sets of $S_\kappa$ of type less than $\kappa$, there exist clopen sets containing the first set and not intersecting the second. Choose one and denote it by $\script{C}_{\pair{A_\alpha,B_\alpha}}$. 

This assignment clearly satisfies $(Cut)$, so it remains to check for $(Mon)$. Let $\beta < \alpha$ and suppose $\pair{A_\beta,B_\beta} \leq_\gamma \pair{A_\alpha,B_\alpha}$ for some $\gamma < \kappa$. Since we chose $\gamma^\downarrow_\beta$ minimal, we have $U_\gamma \subset U_{\gamma_\beta^\downarrow}$. By construction we have $U_{\gamma_\beta^\downarrow} \cap \script{C}_{\pair{A_\beta,B_\beta}} \subset \bigcup \script{C}^\downarrow_\alpha \subset \script{C}_{\pair{A_\alpha,B_\alpha}}$ and therefore also $U_{\gamma} \cap \script{C}_{\pair{A_\beta,B_\beta}} \subset \script{C}_{\pair{A_\alpha,B_\alpha}}.$ Thus, $\script{C}_{\pair{A_\beta,B_\beta}} \subset_\gamma \script{C}_{\pair{A_\alpha,B_\alpha}}$. 

The case $\pair{A_\beta,B_\beta} \geq_\gamma \pair{A_\alpha,B_\alpha}$ is similar and the proof is complete.
\end{proof}

We now consider a variation of the butterfly construction which is tailored to $P_\kappa$-points. Let $\set{U_\alpha}:{\alpha < \kappa}$ be a decreasing neighbourhood base of a $P_\kappa$-point $p$ of $S_\kappa$ with $U_0=S_\kappa$. We work through the ``onion rings'' $D_\alpha = U_\alpha \setminus U_{\alpha+1}$ and assign them either to the $A$- or the $B$-wing, following certain patterns. When compared to the original butterfly construction in Lemma \ref{butterfly}, this adaptation has the advantage that one does not need to assign cuts at successor stages.

The \emph{support} of a binary sequence $f \colon \kappa \to 2$ is the set $f^{-1}(\singleton{1})$.

\begin{mylem}
\label{cell}
Assume $\kappa=\kappa^{<\kappa}$ and let $p$ be a $P_\kappa$-point of $S_\kappa$. There is a family $\set{A^f}:{f \in 2^{\kappa}}$ of clopen subsets of $S_\kappa \setminus \singleton{p}$ such that for all $f,g \in 2^{\kappa}$
\begin{enumerate}
\item $A^f$ is non-empty whenever $f$ has non-empty support,
\item $A^f$ is non-compact whenever $f$ has unbounded support,
\item $A^{1-f} \cap A^f = \emptyset$,
\item if $f \leq g$ (pointwise) then $A^f \subset A^g$ and
\item if $f= g$ eventually then there exists a clopen neighbourhood $U$ of $p$ such that $A^f \cap U = A^g \cap U$.
\end{enumerate}
\end{mylem}

\begin{proof}
Let $\set{U_\alpha}:{\alpha < \kappa}$ be a decreasing neighbourhood base of $p$ with $U_0=S_\kappa$, and let $D_\alpha = U_\alpha \setminus U_{\alpha+1}$. Let $\script{C}$ denote a fixed strong monotone cut operator from Lemma \ref{monotone2} with respect to $\Set{U_\alpha}$. For each sequence $f \in 2^{\kappa}$ we build a butterfly with wings $A^f$ and $B^f$ around $p$. As always, this involves defining $S_\kappa$-clopen sets $A^f_\alpha$ and $B^f_\alpha$ for $\alpha < \kappa$ such that all $(A^f_\alpha,B^f_\beta)$-pairs are disjoint and $A^f_\alpha \cup B^f_\alpha$ covers $S_\kappa \setminus U_\alpha$ for all $\alpha$. The rules for the recursive construction are: for all ordinals $\alpha<\kappa$ set 
$$ A^f_{\alpha+1} = \begin{cases} A^f_\alpha \cup D_\alpha, & \text{if } f(\alpha)=1, \\ A^f_\alpha, & \text{if } f(\alpha)=0, \end{cases} 
\quad \quad 
B^f_{\alpha+1} = \begin{cases} B^f_\alpha, & \text{if } f(\alpha)=1, \\ B^f_\alpha \cup D_\alpha, & \text{if } f(\alpha)=0, \end{cases}$$
and if $\lambda < \kappa$ is a limit ordinal, put 
$$A^f_{\lambda} = \cut{\bigcup_{\beta < \lambda} A^f_\beta, \bigcup_{\beta < \lambda} B^f_\beta} \setminus U_\lambda 
\quad \text{and} \quad
B^f_{\lambda} = (S_\kappa \setminus A^f_\lambda) \setminus U_\lambda.
$$
The sets $A^f = \bigcup_{\alpha < \kappa} A^f_\alpha$ and $B^f = \bigcup_{\alpha < \kappa} B^f_\alpha$ are disjoint open and cover all of $S_\kappa \setminus \singleton{p}$. Thus, they define clopen subsets of $S_\kappa \setminus \singleton{p}$. 

We claim the sets $A^f$ satisfy assertions (1)-(5). It is clear that (1) is satisfied. Next, if $f$ has unbounded support, then $A^f$  limits onto $p$, i.e.\ is non-compact.

For (3) and (4), one shows by induction that $A^{1-f}_\alpha = B^f_\alpha$ and  $A^f_\alpha \subset A^g_\alpha$ whenever $f \leq g$, using that the cut operator is symmetric and monotone, respectively. 

For (5) suppose there exists an ordinal $\delta<\kappa$ such that $f(\alpha) = g(\alpha)$ for all $\alpha \geq \delta$. We show by induction that $A^f_\alpha \cap U_\delta = A^g_\alpha \cap U_\delta$. The claim is trivially true for $\alpha < \delta$. So let $\alpha \geq \delta$ and assume the claim holds for all smaller ordinals. 
The situation is clear for successors, so assume that $\alpha$ is a limit. By induction hypothesis, the pairs 
$\pair{\bigcup_{\beta < \alpha} A^f_{\beta},\bigcup_{\beta < \alpha} B^f_{\beta}}$ and $\pair{\bigcup_{\beta < \alpha} A^g_{\beta},\bigcup_{\beta < \alpha} B^g_{\beta}}$ are $\delta$-equivalent and hence it follows from the properties of the cut operator that
\begin{eqnarray*}
A^f_\alpha \cap U_\delta &=& \script{C}(\bigcup_{\beta < \alpha} A^f_{\beta},\bigcup_{\beta < \alpha} B^f_{\beta}) \cap (U_\delta \setminus U_\alpha) \\
&=& \script{C}(\bigcup_{\beta < \alpha} A^g_{\beta},\bigcup_{\beta < \alpha} B^g_{\beta}) \cap (U_\delta \setminus U_\alpha) = A^g_\alpha \cap U_\delta.
\end{eqnarray*}
This completes the induction step and the proof.
\end{proof}

We now show how to use a family with properties (1)--(5) of Lemma \ref{cell} to push ultrafilters and almost disjoint families from $\kappa$ through to the space $(S_\kappa \setminus \singleton{p})^*$. For a subset $U$ of $\kappa$, let $\mathds{1}_U \in 2^{\kappa}$ denote its characteristic function.

\begin{mythm}
\label{cardinality}
Assume $\kappa=\kappa^{<\kappa}$. For all $x \in S_\kappa$ the space $(S_\kappa \setminus \singleton{x})^*$ has cardinality $2^{2^{\kappa}}$.
\end{mythm}

\begin{proof}
The upper bound is clear by \cite[3.5.3]{Eng}. Let $p$ be a $P_\kappa$-point of $S_\kappa$. By Lemma \ref{wlog} it suffices to show that there are at least $2^{2^{\kappa}}$-many z-ultrafilters on $S_\kappa \setminus \singleton{p}$. 

Let $\set{A^f}:{f \in 2^{\kappa}}$ be a family of clopen sets of $S_\kappa \setminus \singleton{p}$ with properties (1), (3) and (4) of Lemma \ref{cell}. For an ultrafilter $\script{U}$ on $\kappa$, consider the family 
$$\Phi(\script{U}) = \set{A^{\mathds{1}_U}}:{U \in \script{U}}.$$
By (1) and (4), $\Phi(\script{U})$ is a filter base for some clopen filter on $S_\kappa \setminus \singleton{p}$. 

Let us see that if $\script{U}$ and $\script{U}'$ are distinct ultrafilters on $\kappa$, then $\Phi(\script{U})$ and $\Phi(\script{U}')$ can only be extended to distinct z-ultrafilters on $S_\kappa \setminus \singleton{p}$. Indeed, there is $U \subset \kappa$ such that $U \in \script{U}$ and $\kappa \setminus U \in \script{U}'$.  
By (3), $A^{\mathds{1}_U}$ has empty intersection with $A^{\mathds{1}_{\kappa \setminus U}}$, and therefore $\Phi(\script{U})$ and $\Phi(\script{U}')$ cannot be extended to the same z-ultrafilter.

As there are $2^{2^{\kappa}}$ ultrafilters on $\kappa$ the result follows. 
\end{proof}

\begin{mythm}
\label{cellularity}
Assume $\kappa=\kappa^{<\kappa}$. For all $x \in S_\kappa$ the space $(S_\kappa \setminus \singleton{x})^*$ contains a family of $2^{\kappa}$ many disjoint clopen sets. 
\end{mythm}

\begin{proof}
Let $p$ be a $P_\kappa$-point of $S_\kappa$. By Lemma \ref{wlog} it suffices to prove the theorem for $(S_\kappa \setminus \singleton{p})^*$. 

Let $\set{A^f}:{f \in 2^{\kappa}}$ be a family of clopen sets of $S_\kappa \setminus \singleton{p}$ with properties (2)--(5) of Lemma \ref{cell}. Recall that $\kappa=\kappa^{<\kappa}$ if and only if $\kappa$ is regular and $\kappa = 2^{<\kappa}$ \cite[1.27]{Ultrafilters}. In particular, from $\kappa=2^{<\kappa}$ we may conclude that there is an almost disjoint family $\script{E}$ of size $2^{\kappa}$ on $\kappa$, i.e. a family whose members are subsets of $\kappa$ of full cardinality such that the intersection of any two elements is of size less than $\kappa$ \cite[{II.1.3}]{Kunen}. By property (2), the family
$$\Phi(\script{E})=\set{A^{\mathds{1}_E}}:{E \in \script{E}}$$ consists of non-compact clopen sets of $S_\kappa \setminus \singleton{p}$. We claim they have pairwise compact intersection. By regularity of $\kappa$, the intersection of distinct elements $E$ and $F$ of $\script{E}$ is bounded by an ordinal $\delta<\kappa$ . The function
$$f \colon \kappa \to 2, \; \alpha \mapsto \begin{cases} \mathds{1}_F(\alpha) & \text{if } \alpha \leq \delta, \\ 1-\mathds{1}_E(\alpha) & \text{if } \alpha > \delta, \end{cases}$$
satisfies $\mathds{1}_F \leq f$, and $f=1-\mathds{1}_E$ eventually. By property (5), there is a clopen neighbourhood $U$ of $p$ such that $A^f \cap U = A^{1-\mathds{1}_E} \cap U$. Then 
$$A^{\mathds{1}_E} \cap A^{\mathds{1}_F} \cap U \subset A^{\mathds{1}_E} \cap A^{f} \cap U = A^{\mathds{1}_E} \cap A^{1-\mathds{1}_E} \cap U = \emptyset,$$
where the first inclusion follows from property (4) and the last equality from (3). Thus, $A^{\mathds{1}_E} \cap A^{\mathds{1}_F}$ is a closed subset of the compact space $S_\kappa \setminus U$, hence compact.

It follows from \cite[2.6d]{Ultrafilters} that whenever $A$ and $B$ are distinct elements in $\Phi(\script{E})$ then $A^*=\closure{A}\setminus A$ and $B^*$ are disjoint non-empty clopen subsets of $(S_\kappa \setminus \singleton{p})^*$. Since $\Phi(\script{E})$ has cardinality $2^{\kappa}$ the result follows.
\end{proof}

In fact, if $X$ is any compact zero-dimensional $F_\kappa$-space with the property $\kappa=\kappa^{<\kappa}$, and $p \in X$ is a $P_\kappa$-point of character $\kappa$ then the above methods show that $(X \setminus \singleton{p})^*$ has cardinality at least $2^{2^\kappa}$, and contains a family of $2^\kappa$ many disjoint clopen sets.

\section{\bf Open Questions}
\label{section9}
We list open questions that are motivated by results in this paper. Corollary \ref{chchchch} shows that it is consistent with CH that all remainders of the form $(\wstar \setminus \singleton{x})^*$ are homeomorphic. Is the negation of this result also consistent with CH?

\begin{myquest}
Is it consistent with CH that for a $P$-point $p$ and a non-$P$-point $x$ of $\wstar$ the remainders of $\wstar \setminus \singleton{p}$ and $\wstar \setminus \singleton{x}$ are non-homeomorphic?
\end{myquest}

Theorem \ref{thm1.2} shows that it is consistent with ZFC that every space $\wstar \setminus \singleton{x}$ has a one-point Stone-\v{C}ech remainder. Under CH, Theorem \ref{cardinality} shows that the remainder of $\wstar \setminus \singleton{x}$ has size $2^{2^{\w_1}}$ for every $x$. 

\begin{myquest}[van Mill]
\label{Q3}
Which cardinalities, apart from $1$ and $2^{2^{\w_1}}$, can $(\wstar \setminus \singleton{x})^*$ consistently have? 
\end{myquest}

\begin{myquest}
Is it consistent that  $(\wstar \setminus \singleton{x})^*$ and  $(\wstar \setminus \singleton{y})^*$ have different cardinalities for different $x$ and $y$?
\end{myquest}

In Section \ref{section6} we mention the question whether one can prove without CH that $\wstar \setminus \singleton{x}$ is a strongly zero-dimensional $F$-space.
\begin{myquest}
\label{Q4}
Is it a ZFC theorem that $\beta (\wstar \setminus \singleton{x})$ is a Parovi\v{c}enko space? 
\end{myquest}
An affirmative answer to Question \ref{Q4} proves that $(\wstar \setminus \singleton{x})^*$ is a compact $F$-space and hence rules out infinite cardinalities smaller than $2^\cont$ in Question \ref{Q3}. 
%

Lastly, Theorem \ref{hui} gives us a large class of topologically distinct spaces whose remainders are $\kappa^+$-Parovi\v{c}enko spaces. Can we find a precise description of spaces with that behaviour?
\begin{myquest}
Is there a characterisation for which spaces $X$ its remainder $X^*$ is a $\kappa^+$-Parovi\v{c}enko space?
\end{myquest}

\bibliographystyle{plain}

\end{document}